\documentclass[11pt]{amsart}
\usepackage{geometry}
\usepackage{amsmath}
\usepackage{amsfonts}
\usepackage{amsthm}
\usepackage{amssymb}
\allowdisplaybreaks[2]

\newcommand{\equstart}{\begin{equation}\begin{aligned}}
\newcommand{\equend}{\end{aligned}\end{equation}}
\newcommand{\equstartu}{\begin{equation*}\begin{aligned}}
\newcommand{\equendu}{\end{aligned}\end{equation*}}
\newtheorem{theorem}{Theorem}
\newtheorem{lemma}{Lemma}

\theoremstyle{definition}

\newtheorem{question}{Question}
\theoremstyle{remark}

\numberwithin{equation}{section}
\theoremstyle{plain}
\newtheorem{other}{Theorem}         
\newtheorem{otherp}[other]{Proposition}     

\newenvironment{Pf}{\noindent{\emph{Proof of}}}{$\Box$}
\newcommand{\C}{{\mathbb C}}
\newcommand{\D}{{\mathbb D}}
\newcommand{\T}{\mathbb T}

\newcommand{\ol}{\overline}

\newcommand{\re}{\text {Re\,}}
\newcommand{\ig}{\stackrel{\text{def}}{=}}
\newcommand{\hol}{\mathcal Hol}

\begin{document}
\title[Integral means inequalities, convolution, and univalent functions]{Integral means inequalities, convolution, and univalent functions}

\author[D. Girela]{Daniel Girela}
 \address{An\'alisis Matem\'atico,
Universidad de M\'alaga, Campus de Teatinos, 29071 M\'alaga, Spain}
 \email{girela@uma.es}
\author[C. Gonz\'{a}lez]{Crist\'{o}bal Gonz\'{a}lez}
\address{An\'alisis Matem\'atico,
Universidad de M\'alaga, Campus de Teatinos, 29071 M\'alaga, Spain}
 \email{cmge@uma.es}
\thanks{This research is supported by a grant from \lq\lq El Ministerio de Econom\'{\i}a y Competitividad\rq\rq , Spain (MTM2014-52865-P);
and by a grant from la Junta de Andaluc\'{\i}a (FQM-210).}
\thanks{The authors declare that there is no conflict of interest
regarding the publication of this paper.}  \dedicatory{Dedicated to Fernando P\'{e}rez Gonz\'{a}lez on
the occasion of his retirement}
\begin{abstract}
We use the Baernstein star-function to investigate several questions
about the integral means of the convolution of two analytic
functions in the unit disc. The theory of univalent functions plays
a basic role in our work.
\end{abstract}
\maketitle
\section{Introduction}
\par
Let $\D=\{z\in\C:|z|<1\}$ and $\T=\{z\in\C:|z|=1\}$ denote the open
unit disc and the unit circle in the complex plane $\C$. We let also
$\hol(\D)$ be the space of all analytic functions in $\D$ endowed
with the topology of uniform convergence in compact subsets.
\par
If $0\le r<1$ and $f\in\hol (\D)$, we set
\begin{align*}
M_p(r,f)=&\left(\int_{-\pi }^\pi |f(re^{it})|^p\,\frac{dt}{2\pi
}\right)^{1/p},\quad\text{if $
0<p<\infty$},\\M_\infty(r,f)=&\sup_{|z|=r}|f(z)|.\end{align*}
\par
For $0<p\le\infty$, the Hardy space $H^p$ consists of those
$f\in\hol(\D)$ such that
\[\|f\|_{H^p}\ig\sup_{0\le r<1}M_p(r,f)<\infty .\]
We refer to \cite{D} for the theory of $H^p$-spaces.
\par
If $f, g\in \hol (\mathbb D)$,
$$f(z)=\sum_{n=0}^\infty a_nz^n,\quad g(z)=\sum_{n=0}^\infty
b_nz^n\quad (z\in\mathbb D),$$ the (Hadamard) convolution $(f\star
g)$ of $f$ and $g$ is defined by
$$(f\star g)(z)=\sum_{n=0}^\infty a_nb_nz^n,\quad z\in \mathbb D.$$
We have the following integral representation
\begin{equation*}
(f\star g)(z) ={\frac1{2\pi i}\int_{|\xi|=r} f\Bigl(\frac z\xi\Bigr)
g(\xi)\frac{d\xi}{\xi}},\quad \vert z\vert <r<1,\end{equation*} (see
\cite[p.\@\,11]{R}). The convolution operation $\star $ makes $\hol
(\mathbb D)$ into a commutative complex algebra with an identity
\begin{equation}\label{Ident}I(z)=\frac{1}{1-z}=\sum_{n=0}^\infty z^n,\quad z\in \mathbb
D.\end{equation} We refer to \cite{R} for the theory of the
convolution of analytic functions and its connections with geometric
function theory.
\par\medskip Following \cite{Sh73}, we shall say that a function
$F\in \hol (\mathbb D)$ is {\it bound preserving} if for every $f\in
H^\infty $ we have that $f\star F\in H^\infty $ and
$$\Vert f\star F\Vert _{H^\infty }\,\le \Vert f\Vert _{H^\infty
}.$$
\par Sheil-Small \cite[Theorem\,\@1.\,\@3]{Sh73} (see also \cite[p.\,\@123]{R}) proved that a
function $F\in \hol (\mathbb D)$ is bound preserving if and only if
there exists a complex Borel measure $\mu $ on $\mathbb T$ with
$\Vert \mu \Vert \le 1$ such that
\begin{equation*}\label{bp-mu}F(z)=\int _{\mathbb T}\frac{d\mu (\xi
)}{1-z\xi },\quad z\in \mathbb  D.\end{equation*} The measure $\mu $
is a probability measure if and only if $F$ is {\it convexity
preserving}, that is, for any $f\in \hol (\mathbb D)$ the range of
$f\star F$ is contained in the closed convex hull of the range of
$f$ \cite[pp.\,\@123,\,\@124]{R}.
\par It turns out that if $F$ is bound preserving and $1\le p\le
\infty $, then for every $f\in H^p$ we have that $f\star F\in H^p$
and
$$\Vert f\star F\Vert _{H^p}\,\le \Vert f\Vert _{H^p
}.$$ Actually, the following stronger result holds.
\begin{theorem}\label{ThA} Suppose that $f, F\in \hol (\mathbb D)$ with $F$ being
bound preserving. Then \begin{equation}\label{ineq-Mp}M_p(r, f\star
F)\,\le M_p(r, f),\quad 0<r<1,\end{equation} whenever $1\le p\le
\infty $.
\end{theorem}
\begin{proof} Since $F$ is bound preserving, there exists a complex Borel measure $\mu $ on $\mathbb T$ with
$\Vert \mu \Vert \le 1$ such that
\begin{equation*}\label{bp-mu}F(z)=\int _{\mathbb T}\frac{d\mu (\xi
)}{1-z\xi }=\,\sum_{n=0}^\infty \left (\int_{\mathbb T}\xi^n\,d\mu
(\xi)\right )z^n,\quad z\in \mathbb D.\end{equation*} If
$f(z)=\sum_{n=0}^\infty a_nz^n$ ($z\in \mathbb D$), we have
\begin{align*}(f\star F)(z) & =\sum_{n=0}^\infty a_n\left
(\int_{\mathbb T}\xi^n\,d\mu (\xi )\right )z^n\\ &=\,\int_{\mathbb
T}\left (\sum_{n=0}^\infty a_n\xi^nz^n\right )\,d\mu (\xi )\,=\,
\int_{\mathbb T}f(\xi z)\,d\mu (\xi ),\quad z\in\mathbb
D.\end{align*} This immediately yields (\ref{ineq-Mp}) for $p=\infty
$. Now, if $1\le p<\infty $, using Minkowski's integral inequality
we obtain
\begin{align*} M_p(r, f\star F)\,=\,&\left [\frac{1}{2\pi
}\int_{-\pi }^\pi \left \vert \int_{\mathbb T}f(r\xi e^{i\theta
})d\mu (\xi )\right \vert ^pd\theta \right ]^{1/p}\\ \le \,&\left
[\frac{1}{2\pi }\int_{-\pi }^\pi \left (\int_{\mathbb T}\vert f(r\xi
e^{i\theta })\vert
d\,\vert \mu \vert (\xi )\right ) ^pd\theta \right ]^{1/p}\\
\le \,&\int_{\mathbb T}\left (\frac{1}{2\pi }\int_{-\pi }^\pi \vert
f(r\xi e^{i\theta })\vert ^p\,d\theta \right )^{1/p}d\,\vert \mu
\vert (\xi )\\=\,&\int_{\mathbb T}M_p(r, f)\,d\,\vert \mu \vert (\xi
)\,\le \,M_p(r, f).
\end{align*}
\end{proof}
\par\medskip
\section{Star-type inequalities}

The main purpose of this article is studying the possibility of
extending Theorem\,\@\ref{ThA} to cover other integral means, at
least for some special classes of functions. In order to do so, we
shall use the method of the star-function introduced by
A.~Baernstein \cite{Ba73,Ba74}.
\par
If $u$ is a subharmonic function in $\mathbb D\setminus \{0 \}$, the
function $u^* $ is defined by
 \[
    u^*(re^{i\theta})
    =\sup_{|E|=2\theta }\int_E u(re^{it})dt,
    \qquad 0<r<1,\quad 0\leq\theta\leq \pi,
    \]
    where $\vert E\vert $ denotes the Lebesgue measure of the set
    $E$.
The basic properties of the star-function which make it useful to
solve extremal problems are the following \cite{Ba74}:
\begin{itemize}\item If $u$ is a subharmonic function in $\mathbb
D\setminus \{0 \}$, then the function $u^*$ is subharmonic in
$\mathbb D^+=\{ z=re^{i\theta } : 0<r<1, 0<\theta <\pi \} $ and
continuous in $\{ z=re^{i\theta } : 0<r<1, 0\le \theta \le\pi \} $.
\item  If $v$ is harmonic in $\mathbb D\setminus \{0 \}$, and it is a symmetric
decreasing function on each of the circles $\{ \vert z\vert =r\} $
($0<r<1$), then $v^*$ is harmonic in $\mathbb D^+$ and, in fact,
$v^*(re^{i\theta })=\int_{-\theta }^\theta v(re^{it})dt$.
\end{itemize}

\par\medskip The relevance of the star-function to obtain integral
means estimates comes from the following result.
\begin{otherp}[\cite{Ba74}]\label{Pr-1} Let $u$ and $v$ be two subharmonic functions in
$\mathbb D$. Then the following two conditions are equivalent:
\begin{itemize}\item[(i)] $u^*\le v^*$ in $\mathbb D^+$.
\item[(ii)] For every convex and increasing function $\Phi :\mathbb R\rightarrow \mathbb R$, we have that $$\int_{-\pi }^\pi \Phi\left (u(re^{i\theta })\right
)\,d\theta \,\le \,\int_{-\pi }^\pi \Phi\left (v(re^{i\theta
})\right )\,d\theta ,\quad 0<r<1.$$
\end{itemize}
\end{otherp}
Proposition~\ref{Pr-1} yields the following result about analytic
functions.
\begin{otherp}\label{Pr-2} Let $f$ and $g$ be two non-identically zero analytic functions
in $\mathbb D$. Then the following conditions are equivalent:
\begin{itemize}\item[(i)] $\left (\log \vert f\vert \right )^* \le \left (\log \vert g\vert \right )^* $ in $\mathbb D^+$.
\item[(ii)] For every convex and increasing function $\Phi :\mathbb R\rightarrow \mathbb R$, we have that $$\int_{-\pi }^\pi \Phi\left (
\log \vert f(re^{i\theta })\vert\right )\,d\theta \,\le \,\int_{-\pi
}^\pi \Phi\left (\log \vert g(re^{i\theta })\vert\right )\,d\theta
,\quad 0<r<1.$$
\end{itemize}
\end{otherp}
\par Since for any $p>0$ the function $\Phi $ defined by $\Phi
(x)=\exp (px)$ ($x\in \mathbb R$) is convex and increasing we deduce
that if $f$ and $g$ are as in Proposition~\ref{Pr-2} and $\left
(\log \vert f\vert \right )^* \le \left (\log \vert g\vert \right
)^* $ in $\mathbb D^+$, then
$$M_p(r, f)\,\le \,M_p(r, g),\quad 0<r<1,$$ for all $p>0$.
\par The main achievement in the use of the star-function by A. Baernstein
in \cite{Ba74}, was the proof that the Koebe function
$k(z)=\frac{z}{(1-z)^2}$ ($z\in \mathbb D$) is extremal for the
integral means of functions in the class $S$ of univalent functions
(see \cite{D} and \cite{Po:uni} for the notation and results
regarding univalent functions). Namely, Baernstein proved that if
$f\in S$ then
$$\left (\pm\log \vert f\vert\right )^* \,\le \,\left (\pm\log \vert k\vert\right
)^* $$ and, hence,
 \[
\int_{-\pi}^\pi \bigl|f(re^{i\theta})\bigr|^p d\theta \leq
\int_{-\pi}^\pi \bigl|k(re^{i\theta})\bigr|^p d\theta ,\quad 0<r<1,
    \]
for all $p\in \mathbb R$. In particular, we have that if $f\in S$
and $0<p\le \infty $,  then $$M_p(r, f)\le M_p(r, k),\quad 0<r<1.$$
\par Subsequently the star-function has been used in a good number of papers
to obtain bounds on the integral means of distinct classes of
analytic functions (see, e.\,\@g.,
\cite{Ba78,Le79,Br81,Gi86,Gi91,No91}).
\par\medskip
Coming back to convolution, the following questions arise in a
natural way.
\begin{question}\label{q1} Let $f, g, F, G$ be analytic functions in
$\D$ with $|F|$ and  $|G|$ being symmetric decreasing on each of the
circles $\{ \vert z\vert =r\} $ and suppose that
    \[
    \bigl(\log|f|\bigr)^*\leq\bigl(\log|F|\bigr)^*
    \qquad\text{and}\qquad
    \bigl(\log|g|\bigr)^*\leq\bigl(\log|G|\bigr)^*.
    \]
Does it follow that $\bigl(\log|f\star
g|\bigr)^*\leq\bigl(\log|F\star G|\bigr)^*?$
\end{question}

\begin{question}\label{q2} Let $F$ and $f$ be two analytic functions in $\D$ and
suppose that $F$ is bound preserving. Can we assert that $\left
(\log \vert f\star F\vert \right )^*\,\le \,\left (\log \vert f\vert
\right )^*$?
\end{question}
\par\medskip
We shall show that the answer to these two questions is negative.
Regarding Question\,\@\ref{q1} we have the following result.
\begin{theorem}\label{q1-neg} There exist two functions $F_1, F_2\,\in \,\hol (\mathbb D)$ with
\[ \bigl(\log|F_j|\bigr)^*\leq\bigl(\log|I|\bigr)^*,
    \qquad\text{for $j\,=\,1, 2$,}\]
and such that
\begin{equation}\label{no-star}\text{the inequality\, $\bigl(\log|F_1\star
F_2|\bigr)^*\leq\bigl(\log|I\star I|\bigr)^*$ does not
hold.}\end{equation}
\par Here,  $I$ is the identity element of the convolution defined
in (\ref{Ident}), that is, $I(z)=\frac{1}{1-z}$ ($z\in \mathbb D$).
Hence $I\star I=I$.
\end{theorem}
\begin{proof} Let $h$ be an odd function in the class $S$ with
Taylor expansion
$$h(z)\,=\,z\,+\,a_3z^3\,+\,a_5z^5\,+\,\dots $$ with $\vert a_5\vert >1$. The
existence of such an $h$ was proved by Fekete and Szeg\"{o} (see
\cite[p.\,\@104]{D.Univ}). Set also
\begin{equation}\label{h1}h_1(z)=\frac{h(z)}{z}\,=\,1\,+\,a_3z^2\,+\,a_5z^4\,+\,\dots
,\quad z\in \mathbb D.\end{equation} It is well known that there
exists a function $H\in S$ such that $h(z)=\sqrt {H(z^2)}$ (see
\cite[p.\,\@64]{D.Univ}). Set $k_2(z)=\sqrt
{k(z^2)}=\frac{z}{1-z^2}$ and
$J(z)=\frac{k_2(z)}{z}=\frac{1}{1-z^2}$ ($z\in \mathbb D$). By
Baernstein's theorem we have $\left (\log \vert H\vert \right )^*\le
\left (\log \vert k\vert \right )^*$, a fact which easily implies
that $\left(\log\vert h_1\vert \right )^*\,\le \,\left (\log\vert
J\vert \right )^*$. Now, it is clear that $J$ is subordinate to $I$
and then, using \cite[Lemma\,\@2]{Le79}, we see that $\left
(\log\vert J\vert \right )^*\,\le\,\left (\log\vert I\vert \right
)^*.$ Thus it follows that
\begin{equation}\label{fI}\left (\log\vert h_1\vert \right
)^*\,\le\,\left (\log\vert I\vert \right )^*.\end{equation} For
$n=1, 2, 3, \dots $, we define $f_n$ inductively  as follows
$$f_1=h_1\quad\text{and\quad $f_n=f_{n-1}\star f_1,$\, for $n\ge 2$.}$$
In other words, $f_n=\overbrace{h_1\star\dotsb\star h_1}^{(n)}$.
Clearly, (\ref{h1}) yields
$$f_n(z)\,=\,1\,+\,a_3^nz^2\,+\,a_5^nz^4\,+\,\dots .$$
Since $\vert a_5\vert >1$, it follows that $\vert a_5^n\vert \to
\infty $, as $n\to \infty $. This is equivalent to saying that
$$\vert f_n^{(4)}(0)\vert \,\to \,\infty ,\quad \text{as $n\to \infty
$}.$$ Then it follows that the family $\{ f_n^{(4)} : n=1, 2, 3,
\dots \} $ is not a locally bounded family of holomorphic functions
in $\mathbb D$. Using \cite[Theorem\,\@16, p.\,\@225]{Ahlfors} we
see that the same is true for the family $\{ f_n : n=1, 2, 3, \dots
\} $. Take $p\in (0, 1)$, then $I\in H^p$. Since a bounded subset of
$H^p$ is a locally bounded family \cite[p.\,\@36]{D}, it follows
that
\begin{equation}\label{supHp}\sup_{n\ge 1}\Vert f_n\Vert
_{H^p}\,=\,\infty .\end{equation} Now, (\ref{supHp}) implies that
$\Vert f_n\Vert _{H^p}> \Vert I\Vert _{H^p}$ for some $n$. Using
Proposition\,\@\ref{Pr-2}, we see that this implies that
\begin{equation*}\label{star-NO}\text{the inequality $\left (\log \vert f_n\vert \right
)^*\le \left (\log \vert I\vert \right )^*$ is not true for some
$n$.}\end{equation*} Let $N$ be the smallest of all such $n$. Using
(\ref{fI}) and the fact that $f_1=h_1$, it follows that that $N>1$.
\par Then it is clear that (\ref{no-star}) holds with $F_1=f_1, F_2=f_{N-1}$.
\end{proof}\par\medskip
We have the following result regarding Question\,\@\ref{q2}.
\begin{theorem}\label{q2-neg} There exist $f, F$ analytic and univalent
in $\mathbb D$ such that $F$ is convexity preserving and with the
property that the inequality $\left (\log \vert f\star F\vert \right
)^*\le \left (\log \vert f\vert \right )^*$ does not hold.
\end{theorem}
\par The following lemma will be used in the proof of
Theorem\,\@\ref{q2-neg}.
\begin{lemma}\label{lem-no}
Let $f, F\in \hol (\mathbb (D)$ and suppose that $F(0)=1$, $F$ is
convexity preserving, and that $f$ and $f\star F$ are zero-free in
$\mathbb D$ and satisfy the inequality $\left (\log \vert f\star
F\vert \right )^*\le \left (\log \vert f\vert \right )^*.$ Then we
also have that
\begin{equation}\label{ine-inv}\left (\log \left \vert \frac{1}{f\star F}\right \vert \right )^*\le \left (\log \left \vert \frac{1}{f}\right \vert \right
)^*.\end{equation} \end{lemma}
\begin{proof}
Set $u=\log \vert f\star F\vert $, $v=\log \vert f\vert $. Then $u$
and $v$ are harmonic in $\mathbb D$, $u(0)=v(0)$, and $u^*\le v^*$.
Then it follows that, for $0<r<1$ and $0\le \theta \le \pi $,
\begin{align*}(-u)^*(re^{i\theta })\,=\,&\sup_{\vert E\vert =2\theta
}\int _{E}-u(re^{it})dt\\=\,&\sup_{\vert E\vert =2\theta }\left
(-\int _{-\pi }^\pi u(re^{it})dt\,+\,\int_{[-\pi ,\pi ]\setminus
E}u(re^{it})dt\right )\\ =&-2\pi u(0)\,+\,u^*(re^{i(\pi -\theta
)})\,=\,-2\pi v(0)\,+\,u^*(re^{i(\pi -\theta )})\\ \le & -2\pi
v(0)\,+\,v^*(re^{i(\pi -\theta )})\,=\,(-v)^*(re^{i\theta
}).\end{align*} Hence, we have proved that $(-u)^*\le (-v)^*$ which
is equivalent to (\ref{ine-inv}).
\end{proof}

\begin{Pf}{\,\em{Theorem\,\@\ref{q2-neg}.}} Set
$$f(z)=\frac{1}{(1-z)^2 }=\sum_{n=0}^\infty (n+1)z^n,\quad F(z)\,=\,1\,-\,\frac{1}{2}z, \quad z\in \mathbb D.$$
Clearly, $f$ and $F$ are analytic, univalent, and zero-free in
$\mathbb D$. Also
\begin{equation*}\label{fff}(f\star F)(z)=1-z,\quad z\in \mathbb
D.\end{equation*} Hence $f\star F$ is also zero-free in $\mathbb D$.
Notice that $\frac{1}{f\star F}\not\in H^\infty $ \, and
$\frac{1}{f}\in H^\infty $. Then it follows that
\begin{equation}\label{infff}\text{the inequality $\left (\log \left\vert \frac{1}{f\star F}\right \vert \right )^*\le \left (\log
\left \vert\frac{1}{f}\right \vert \right )^*$ does not
hold.}\end{equation} Now, it is a simple exercise to check that
$$F(z)=\frac{1}{2\pi }\int _{-\pi }^\pi \frac{1-\cos \theta
}{1-e^{i\theta }z}\,d\theta $$and then it follows that $F$ is
convexity preserving. Then,  using (\ref{infff}) and
Lemma\,\@\ref{lem-no}, it follows that the inequality $\left (\log
\vert f\star F\vert \right )^*\le \left (\log \vert f\vert \right
)^*$ does not hold, as desired. \end{Pf}
\par\medskip
We close the paper with a positive result, determining a class of
univalent functions $\mathcal Z$ such that (\ref{ineq-Mp}) is true
for all $p>0$, whenever $f\in \mathcal Z$ and $F$ is convexity
preserving.
\par A domain $D$ in $\mathbb C$ is said to be Steiner symmetric if
its intersection with each vertical line is either empty, or is the
whole line, or is a segment placed symmetrically with respect to the
real axis. We let $\mathcal Z$ be the class of all functions $f$
which are analytic and univalent in $\mathbb D$ with $f(0)=0$,
$f^\prime (0)>0$, and whose image is a Steiner symmetric domain. The
elements of $\mathcal Z$ will be called Steiner symmetric functions.
Using arguments similar to those used by Jenkins \cite{Jen} for
circularly symmetric functions, we see that a univalent function $f$
with $f(0)=0$ and $f^\prime (0)>0$ is Steiner symmetric if and only
if it satisfies the following two conditions: (i) $f$ is typically
real and (ii) $\re f$ is a symmetric decreasing function on each of
the circles $\{ \vert z\vert =r\}$ ($0<r<1$). Then it follows that
if $f\in \mathcal Z$ then for every $r\in (0, 1)$, the domain
$f\left (\{ \vert z\vert <r\} \right )$ is a Steiner symmetric
domain and, hence, the function $f_r$ defined by $f_r(z)=f(rz)$
\,($z\in \mathbb D$) belongs to $\mathcal Z$ and it extends to an
analytic function in the closed unit disc $\overline {\mathbb D}$.
Now we can state our last result.
\begin{theorem} Suppose that $f\in\mathcal Z$ and let $F$ be an analytic
function in $\mathbb D$ which is convexity preserving. We have, for
every $p>0$,
\begin{equation}\label{ineq-Mp-V}M_p(r, f\star
F)\,\le M_p(r, f),\quad 0<r<1.\end{equation}
\end{theorem}
\begin{proof} In view of Theorem\,\@\ref{ThA} we only need to prove
(\ref{ineq-Mp-V}) for $0<p<1$. Let $\mu $ be the probability measure
on $\mathbb T$ such that $F(z)=\int _{\mathbb T}\frac{d\mu (\xi
)}{1-z\xi }$\, ($z\in \mathbb  D$). Then we have
\begin{equation}\label{rep}\left (f\star F\right )(z)=\int_{\mathbb T}f(\xi
z)d\mu (\xi ).\end{equation} Since $F$ is convexity preserving, for
$0<r<1$, we have  that $\left (f_r\star F\right )(\overline {\mathbb
D})$ is contained in the closed convex hull of $f_r(\ol\D)$. This
easily yields
\begin{equation*}
    \min_{z\in\D}\re f_r(z)
    \leq \min_{z\in\D}\re(f_r\star F)(z),\quad
    \max_{z\in\D}\re(f_r\star F)(z)
    \leq \max_{z\in\D}\re f_r(z).
    \end{equation*}
    By the remarks in the previous paragraph, we find that, for all $r\in (0, 1)$, $f_r$ belongs to $\mathcal Z$ and
    extends to an analytic function in the closed unit disc
$\overline {\mathbb D}$.  Finally, we claim that
\begin{equation}\label{last}\bigl(\re(f_r\star F)\bigr)^*\leq
\bigl(\re
    f_r\bigr)^*,\quad 0<r<1.\end{equation} Once this is proved, using Proposition\,\@6 of
    \cite{Gi86}, we deduce that
    $$M_p(r, f\star F)=\Vert f_r\star F\Vert _{H^p}\le\Vert f_r\Vert
    _{H^p}=M_p(r, f),\quad 0<p\le 2,$$ finishing our proof.
\par So we proceed to prove (\ref{last}). Fix $r\in (0, 1)$ and set
$u=\re (f_r\star F)$, $v=\re f_r$. Using (\ref{rep}), we have, for
$0<R<1$ and $0<\theta <\pi $,
\begin{align*}
    u^*(Re^{i\theta})
    &=\sup_{|E|=2\theta}\int_E u(Re^{it})dt
    =\sup_{|E|=2\theta}\int_E
        \int_{\T} v(Re^{it}\xi)d\mu(\xi) dt\\
    &=\sup_{|E|=2\theta }\int_{\T}\int_E
        v(Re^{it}\xi)dt d\mu(\xi)
    \leq\int_{\T}v^*(Re^{i\theta })d\mu(\xi)
    =v^*(Re^{i\theta }).
    \end{align*}
    \end{proof}
    \par\medskip
    {\bf Acknowledgements.} We wish to express our gratitude to the
    referee who made valuable suggestions for improvement.
    \medskip


\end{document}